\documentclass[a4paper]{amsart}

\usepackage{graphicx}
\usepackage[utf8]{inputenc}
\usepackage[english]{babel}
\usepackage{braket}

\usepackage{amssymb}
\usepackage{amsmath}
\usepackage{upgreek}
\usepackage{multicol}
\usepackage{enumerate}
\usepackage{float}
\usepackage{mathrsfs} %para letras cursivas
\usepackage[dvipsnames]{xcolor}
\usepackage{comment}
\usepackage{bbm}

\usepackage{amsthm} %para theoremstyle
\theoremstyle{definition}

\usepackage{amsmath,environ}
\NewEnviron{equsplit}
{
	\begin{equation*}\begin{split}
	\BODY
	\end{split}\end{equation*}
}
\newcommand{\R}{\mathbb{R}}
\newcommand{\N}{\mathbb{N}}
\newcommand{\Z}{\mathbb{Z}}
\newcommand{\C}{\mathbb{C}}

\usepackage{mathtools}
\renewenvironment{cases}{%
 \begin{dcases}%
    }{%
 \end{dcases}%\kern-\nulldelimiterspace%
}
\newcounter{defi}
\numberwithin{defi}{section}

\newtheorem{theorem}[defi]{Theorem}
\newtheorem{proposition}[defi]{Proposition}
\newtheorem{corollary}[defi]{Corollary}
\newtheorem{definition}[defi]{Definition}
\newtheorem{lemma}[defi]{Lemma}
\newtheorem{example}[defi]{Example}
\newtheorem*{theorem*}{Theorem}
\newtheorem*{lemma*}{Lemma}
\newtheorem*{conjecture*}{Conjecture}
\newtheorem*{definition*}{Definition}
\newtheorem{remark}[defi]{Remark}

\newtheorem{mainthm}{Theorem}

\newcommand{\action}{G\curvearrowright X}

\newcommand{\supp}[1]{\text{supp}\left(#1\right)}
\newcommand{\Fix}[1]{\text{Fix}(#1)}
\newcommand{\Aper}{\text{Aper}(G\curvearrowright X)}
\newcommand{\luno}{\ell^1(G\curvearrowright X)}
\newcommand{\csigma}{C^*(G\curvearrowright X)}
\newcommand{\Per}{\text{Per}(G\curvearrowright X)}

\newcommand{\lone}[2]{\ell^1(#1\curvearrowright #2)}

\newcommand{\phixs}[3]{\phi_{#1,#2}(#3)}

\usepackage[dvipsnames]{xcolor}
\usepackage[shortlabels]{enumitem}
\usepackage{tikz-cd}

\title[On the involutive Banach algebra]{On the involutive Banach algebra associated to topologically free dynamical systems}

\date{\today}

\thanks{The author was supported by a Comisión Académica de Posgrado (CAP) Scholarship and by Grupo CSIC 883174 (UdelaR, Uruguay). Aditionally, the author was partially supported by Eusebio Gardella through the Swedish Research Council Grant 2021-04561.}

\author[Tabaré Roland]{Tabaré Roland}
\address{Instituto de Matemática y Estadística, Facultad de Ingeniería, Universidad de la República, Uruguay}
\email{tabarer@fing.edu.uy}

\begin{document}

\begin{abstract}
    Given an action $G \curvearrowright X$ of a discrete and countable infinite group $G$ on a compact and Hausdorff space $X$, we regard $\luno$ as the Banach *-algebra crossed product associated to the action. We characterize topological freeness of the action by showing that it is equivalent to every nontrivial closed ideal of $\luno$ intersecting $C(X)$ nontrivially. Most surprisingly, we show that when $G$ is torsion-free and abelian, $\luno$ can detect freeness of $\action$: indeed, we show that $\action$ is free if and only if every closed ideal of $\luno$ is self-adjoint, a property that is automatic in $C^*$-algebras. We also show with an example that this result does not hold beyond the torsion-free abelian case.
\end{abstract}

\maketitle

\tableofcontents

\section{Introduction}

The crossed product is a usual construction in the theory of operator algebras that allows one to obtain an algebra from the action of a group on a $C^*$-algebra. It is desirable that the crossed product has a Banach algebra structure, and the most common way to do that is by endowing the algebraic skeleton of the algebra with a $C^*$-seminorm in order to take the completion of the algebra, which is then a $C^*$-algebra. $C^*$-crossed products have been studied at length through the years, and are a really important family of examples of $C^*$-algebras.

A concrete family of examples comes from an action $\action$ of a countable infinite discrete group $G$ on a compact Hausdorff space $X$ via homeomorphisms, also called a dynamical system. In this situation, the action of the group on the space gives rise to an action on the algebra $C(X)$ of continuous functions, which is a commutative $C^*$-algebra, and from this context one can construct a crossed product endowed with a $C^*$-norm. Nevertheless, there are many other types of Banach algebras that can be obtained from the action of a group on a $C^*$-algebra, and in the latter years there has been a rise of interest in those. For example, the theory of $L^p$-operator algebras has been sparkling in recent years, and an analogue $L^p$-operator algebra crossed product can be constructed; see, for example, the survey paper \cite{modernLp} by Gardella. One can also consider the involutive Banach algebra $\luno$, a concrete example whose $C^*$-envelope is the usual $C^*$-crossed product. In this paper we are interested in the study of this algebra and, in particular, of its ideal structure in relation with the dynamics of the action it arises from. Here, by an ideal we always mean a two-sided ideal. The algebra $\luno$ has a very rich structure. For example, it can have closed non-self-adjoint ideals, which is something that cannot occur in its $C^*$-envelope because, in a $C^*$-algebra, every closed ideal is self-adjoint. This is one of our main interests, as we are going to see that the existence of closed non-self-adjoint ideals is related to the action not being free. This means that we may detect freeness of the action $\action$ by looking at the algebra $\luno$, in a way that cannot be done in the $C^*$-algebraic crossed product.

For actions of $\Z$, the algebra $\lone{\Z}{X}$ has been studied by de Jeu, Svensson and Tomiyama in \cite{1}, by de Jeu and Tomiyama in \cite{2,3,4,5}, and by Kishimoto and Tomiyama in \cite{KishimotoTomiyama}. More recently, in \cite{gradedAlgebrasAndPartialActions}, Jaur\'e and M\u{a}ntoiu studied the $\ell^1$-Banach *-subalgebra of a topologically graded $C^*$-algebra over a discrete group and, when applying their results to actions of groups, obtained that $\luno$ is inverse closed in its enveloping $C^*$-crossed product if the acting group is rigidly symmetric (see \cite[Proposition 3.1]{gradedAlgebrasAndPartialActions}, where the result is stated in the more general framework of partial actions). We are interested in the type of results obtained in \cite{1}. There, it was shown how different dynamical properties are related to analytic-algebraic properties of the ideal structure of $\lone{\Z}{X}$. In particular, \cite[Theorem 4.4]{1} states the following: an integer action $\Z\curvearrowright X$ is free (i.e., the action of every element of $G$ apart from the unit has no fixed points) if and only if every closed ideal of $\lone{\Z}{X}$ is self-adjoint.

We are interested in detecting freeness of $\action$ by looking at the algebra $\luno$ in the broader context of actions of discrete and countable infinite groups. There are several obstructions to extending the arguments used in \cite[Theorem 4.4]{1} to actions of more general groups. First of all, the proof uses the existence of a closed non-self-adjoint ideal in $\ell^1(\Z)$, which is not necessarily true for an arbitrary countable discrete group. It also uses the facts that, for $\Z$-actions, not being free is equivalent to having a finite orbit, and that every nontrivial subgroup of $\Z$ is isomorphic to $\Z$. As it turns out, the statement of \cite[Theorem 4.4]{1} does not hold in general for arbitrary groups; see Example \refeq{ctrExample}. On the other hand, it can be shown that freeness of the action always implies that every closed ideal is self-adjoint (see Corollary \refeq{libImpSelfAdj}). In order to prove the other implication we delve into the realm of abelian groups. Imposing some conditions on the group, we obtain the following concrete statement for abelian groups.

\begin{mainthm}\label{mainThmD}
    Let $G$  be a torsion-free abelian, discrete and countable group. Then, the action $\action$ is free if and only if every closed ideal of $\luno$ is self-adjoint.
\end{mainthm}

We remark that the analogue of the result above, in the $C^*$-algebraic context, does not hold, as closed ideals are always self-adjoint in $C^*$-algebras. As such, the $C^*$-algebraic crossed product cannot detect freeness of the action the same way that $\luno$ does, and may not even detect freeness at all. This shows a stark contrast between the algebra $\luno$ and its $C^*$-envelope.

Different notions of freeness of an action play a key role in this paper. Apart of freeness of the action, the notion of topological freeness will be recurrent (see Definition \refeq{topFree}). The first landmark that we come across in this paper is the following theorem, in which topological freeness of $\action$ is characterized through an analytic-algebraic property of $\luno$. In order to make sense of the theorem, we mention that $\luno$ canonically contains the algebra $C(X)$ of continuous functions.

\begin{mainthm}\label{mainThmA}
    The following are equivalent:
    \begin{enumerate}[(1)]
        \item For every closed nonzero ideal $I$ of $\luno$, the intersection $I\cap C(X)$ is nonzero.
        \item For every closed and self-adjoint nonzero ideal of $\luno$, the intersection $I\cap C(X)$ is nonzero.
        \item The action $\action$ is topologically free.
    \end{enumerate}
\end{mainthm}

We point out the distinction between closed and self-adjoint ideals and just closed ideals done in (1) and (2): as we mentioned before, in this context it is not automatic that every closed ideal is self-adjoint, so the difference is meaningful. This distinction will be recurrent later in the paper. The equivalence between (1) and (3) can actually be seen as an interpretation for $\luno$ of \cite[Theorem 5.9]{twistedCrossProd}, which is a result on the more general context of twisted Banach algebra crossed products. However, the proof here is in the spirit of \cite[Theorem 4.1]{1}, of which it is a generalization for arbitrary discrete and countable (infinite) groups.

Theorem \refeq{mainThmA} proves to be an extremely useful tool. Assuming topological freeness of $\action$, it allows us to easily see how different dynamical properties of the action are characterized as analytic-algebraic properties of $\luno$. The dynamical properties we characterize are minimality, topological transitivity and residual topological freeness (Definitions \refeq{defMinimality}, \refeq{defTransitivity} and \refeq{resTopFree}, respectively); see Theorems \refeq{thmMinSimp}, \refeq{thmTransitivity} and \refeq{rtfThm}. The characterizations of minimality and topological transitivity generalize \cite[Theorems 4.2, 4.10]{1}, which were proved for $\Z$-actions. In these results, topological freeness of the action is not assumed, because for $\Z$-actions it is automatic given the minimality or topological transitivity. The need to assume topological freeness in the context of arbitrary countable discrete groups does not come as a surprise, given that it was necessary in the context of $C^*$-algebraic crossed products; see, for instance, \cite[Theorem 2]{cstarminimality}.

Some precedents show that algebras of type $L^1$ retain more information than their $C^*$ counterparts. A great example of this phenomenom is Wendel's theorem (see \cite{wendel}), which says that two locally compact Hausdorff groups are isomorphic (as topological groups) if and only if their $L^1$-group algebras are isometrically isomorphic. This means that $L^1$-group algebras retain all the information of the group, something that the $C^*$-group algebras do not satisfy as, for example, the group $C^*$-algebras $C^*(\Z_4)$ and $C^*(\Z_2 \times \Z_2)$ are isomorphic but $\Z_4$ is not isomorphic to $\Z_2\times\Z_2$\footnote{In this paper, given $n\in\N$, we refer to $\Z/n\Z$ (the integers modulo $n$) as $\Z_n$.}. For $C^*$-algebraic crossed products, the algebra does not contain that much information of the action by itself. But, given two topologically free actions (see Definition \refeq{topFree}), an isomorphism between their crossed products that preserves a canonical commutative $C^*$-subalgebra is equivalent to the actions being continuous orbit equivalent (see \cite[Theorem 1.2]{LiOrbitEquivalence}). More rigidity is seen on $L^p$-operator algebra crossed products, for $p\in [1,\infty)\setminus\{2\}$, as for two topologically free actions, an isometric isomorphism between the crossed product automatically preserves the canonical commutative subalgebra of the crossed products. Then, it holds in this context that the crossed products are isometrically isomorphic if and only if the actions are continuous orbit equivalent (see \cite[Theorem 6.7]{rigidityLp}). There are no rigidity results about $\luno$ as of now, but it is clear that all the information of the action retained by its associated $C^*$-crossed product is also retained by $\luno$, as the $C^*$-crossed product is the envelope of $\luno$. We hope for some of the results that we present as hinting that the $\ell^1$-crossed product may be even more rigid than its counterparts described above.

Let us briefly describe the structure of the paper. In Section 2 we explain the context in which we work, introducing the basic definitions and notations we use. In particular, we introduce $\luno$, our main object of study. In Section 3, we do some technical work in order to prove Theorem \refeq{mainThmA}. The characterizations of minimality, topological transitivity and residual topological freeness follow from there. In Section 4, we study how to relate non-freeness of the action with the existence of closed non-self-adjoint ideals in $\luno$. We first prove that if the action is free, then every closed ideal of $\luno$ is self-adjoint. Finally, we show how the existence of a nontrivial point-stabilizer subgroup satisfying some conditions implies the existence of a closed non-self-adjoint ideal in $\luno$ (see Proposition \refeq{propAlmFree}); this result implies Theorem \refeq{mainThmD}.

\subsection*{Acknowledgments} 
The author is deeply thankful to his late advisor Fernando Abadie, whose guidance was essential to make all this work possible.

The following work is part of the master's thesis of the author, who would also like to thank his other advisor Matilde Martínez, and Eusebio Gardella and Hannes Thiel for valuable discussions. Part of this work was done during a research visit to the Department of Mathematical Sciences at Chalmers University and the University of Gothenburg, and the author would like to thank the institute's Operator Algebras group for their hospitality.

\section{Definitions and preliminaries}
In this section we introduce the basic objects playing a role in this work, and we fix the notation that we use. In what follows, $G$ is a countable infinite discrete group and $X$ is a compact Hausdorff space. Given a subset $Y$ of $X$, we denote $\overline{Y}$ its closure in $X$.

\subsection{The algebra $\boldsymbol{\luno}$}
Let $G$ act by homeomorphisms on $X$. For every $s\in G$, we denote by $\sigma_s:X\to X$ the homeomorphism associated with the action of $s$ in $X$. We denote this action by $\action$. We denote by $e$ the identity element of $G$, and we denote by $C(X)= \{f:X\to \C ; f \text{ is continuous}\}$ the $C^*$-algebra of continuous functions on $X$ endowed with the supremum norm $\lVert\>\cdot\>\rVert_\infty$. The action of $G$ on $X$ induces an action of $G$ on $C(X)$, that we denote by $\alpha:G\to \text{Aut}(C(X))$, given by
$$\alpha_s(f)(x) = f(\sigma_s^{-1}(x))$$
for every $s\in G$, every $f\in C(X)$ and every $x\in X$.

We define 
$$\luno = \left\{ f:G\to C(X) : \sum_{s\in G} \lVert f(s) \rVert_\infty < \infty \right\}.$$
Set $\lVert f\rVert = \sum_{s\in G}\lVert f(s)\rVert_\infty$ for every $f\in \luno$. Then $\luno$, with pointwise sum and scalar multiplication, is a Banach space. We endow $\luno$ with the following product: given $f,g\in\luno$, we set
$$(fg)(s) = \sum_{t\in G}f(t)\alpha_t(g(t^{-1}s)) \text{ for every }s\in G.$$
We also define an involution $*:\luno\to\luno$ in the following way: given $f\in\luno$, we set
$$(f^*)(s) = \alpha_s(\overline{f(s^{-1})}) \text{ for every } s\in G.$$
With these product and involution we have that $\luno$ is a Banach *-algebra. We can canonically embed $C(X)$ in $\luno$ by identifying it (isometrically) with the $\ell^1$-functions $G\to C(X)$ supported on the unit $e\in G$.

We now describe a useful way to work with the elements of $\luno$. For every $s\in G$, we define $\delta_s\in\luno$ as:
\begin{equation*}
    \delta_s(s') = 
    \begin{cases}
        1 & \text{if } s' =s,\\
        0 & \text{if } s'\neq s,\\
    \end{cases}
\end{equation*}
where $1$ denotes the function taking constant value $1$. With this notation, we can describe an element $f\in\luno$ as:
$$f = \sum_{s\in G} f(s) \delta_s = \sum_{s\in G}f_s \delta_s,$$
where $f_s$ denotes $f(s)\in C(X)$ for every $s\in G$. We usually use this description of the elements of $\luno$ without explicit mention. With this notation, we have that
$$(f_s\delta_s)(g_t \delta_t) = f_s \alpha_s(g_t) \delta_{st}$$
for every $s,t\in G$, and every $f_s,g_t\in C(X)$. Also, with this notation we have that the involution is given by the rule
$$(f_s \delta_s)^* = \alpha_{s}^{-1}(\overline{f_s}) \delta_{s^{-1}} \text{ for every }s\in G \text{ and }f_s\in C(X).$$

\begin{definition}\label{projE}
    We define the canonical projection $E:\luno \to C(X)$ as the map given by $E(f) = f(e)$ for every $f\in \luno$.
\end{definition}

We denote by $\csigma$ the (universal) $C^*$-envelope of $\luno$, which is the universal crossed product $C^*$-algebra associated to $\action$.

We briefly give some definitions related to the action $\action$. Given $x\in X$, we denote by $G_x$ the \textit{stabilizer subgroup of $x$} as the subgroup of $G$ defined as $G_x = \{ s\in G : \sigma_s(x) = x\}$. Given $s\in G$, we define the set of fixed points of $s$ as $\Fix{s} = \{x\in X : \sigma_s(x) = x\}$. It is clear that it is a closed subset of $X$. We define $\Per  = \{ x\in X : \sigma_s (x) = x \text{ for some } s\in G\setminus \{e\}\}$. If $x$ is in $\Per$, we say that it is \textit{periodic}. It is clear that $x$ belongs to $\Per$ if and only if $G_x$ is not trivial. We define $\Aper$ as $X \setminus \Per$. If $x$ belongs to $\Aper$, we say that it is \textit{aperiodic} or that \textit{it has free orbit}.

\subsection{Representations of $\boldsymbol{\luno}$}\label{secRep} In this paper, by a representation of a Banach *-algebra we mean a bounded homomorphism of the algebra into the bounded operators on a Hilbert space that respects the *-operation. As such, the kernel of a representation is always a closed and self-adjoint ideal.

We introduce a family of representations of $\luno$ that is going to be of help in the following section. In \cite[Chapter 4]{tomiyamaBook1}, given $x\in X$ and a unitary representation of the isotropy subgroup $G_x$, it is shown how to obtain a representation of the universal crossed product $\csigma$. As $\csigma$ and $\luno$ have the same theory of representations, we also obtain (via restriction) a representation of $\luno$. We are interested in the representation, that we denote by $\pi_x$, obtained when the unitary representation of $G_x$ we start with is the trivial representation (on a Hilbert space of dimension 1, that is, the complex numbers $\C$). We concretely describe the representation $\pi_x$ in this case.

So, let's fix $x\in X$ and consider the isotropy subgroup $G_x$. We can describe the left quotient space $G/G_x=\{s_\beta G_x : \beta \in J_x\}$ for the set of representatives $\{s_\beta : \beta\in J_x\}$. We consider the Hilbert space $H_x$ with orthonormal basis $\{v_\beta : \beta \in J_x\}$. Then, the representation $\pi_x:\luno\to B(H_x)$ is given by the following:
\begin{itemize}
    \item Given $f\in C(X)$ and $\beta \in J_x$, we have that $\pi_x(f) v_\beta = f(\sigma_{s_\beta}(x))v_\beta$.
    \item Given $s\in G$ and $\beta \in J_x$, there exists a unique $\gamma\in J_x$ such that $ss_\beta$ belongs to the coset $s_\gamma G_x$. Then, we have that $\pi_x (\delta_s) v_\beta = v_{\gamma}$.
\end{itemize}
As $\pi_x$ is multiplicative (and linear), this determines the representation. Thus, given an element $f = \sum_{s\in G}f_s \delta_s \in \luno$, we have that
$$\pi_x \left( \sum_{s\in G} f_s \delta_s \right) = \sum_{s\in G} \pi_x (f_s) \pi_x(\delta_s) \in B(H_x).$$

We note that the representation is finite dimensional if and only if $x$ has finite orbit.

\section{The ideal intersection property and consequences}
The objective of this section is to obtain a dictionary between dynamical properties of the action $\action$ and analytic-algebraic properties of $\luno$. The main tool in order to do this is Theorem \refeq{topFreeThm}, which is our first objective. Before being able to prove this theorem, we have to do some technical work. Many of the results here are heavily inspired by \cite{1}, but generalized for an arbitrary countable infinite discrete group instead of $\Z$. 

We define the commutant of $C(X)$ as
$$C(X)' = \left\{ f\in\luno : fg =gf \text{ for every } g\in C(X) \right\}.$$
It's clear that $C(X) \subset C(X)'$. We characterize the commutant in the following proposition. We recall that, given $g\in C(X)$, by $\supp{g} = \overline{\{ x\in X : g(x)\neq0\}}$ we denote the support of $g$.

\begin{proposition}\label{charCommutant}
    We have that
    $$C(X)' = \left\{ f = \sum_{s\in G} f_s\delta_s \in \luno : \supp{f_s}\subset \Fix{s} \text{ for every }s\in G\right\}.$$
\end{proposition}
\begin{proof}
    Let $f=\sum_{s\in G}f_s\delta_s \in C(X)'$. This means that, for every $g\in C(X)$, we have that $fg=gf$. We compute directly that $fg = \sum_{s\in G} f_s\alpha_s(g)\delta_s$ and that $gf = \sum_{s\in G} gf_s\delta_s$. As $\alpha_s(g) = g\circ \sigma_s^{-1}$, we must have $f_s (g \circ \sigma^{-1}_s) = f_s g$ for every $s\in G$ and for every $g\in C(X)$. Fix $s\in G$ and let $x\in X$ such that $f_s(x)\neq 0.$ Then we have that $g(\sigma_s^{-1}(x)) = g(x)$ for every $g\in C(X)$. This implies that $x = \sigma_s^{-1}(x)$, which is equivalent to $\sigma_s(x) = x$, i.e., $x$ belongs to $\Fix{s}$. This means that $\supp{f_s}\subset \Fix{s}$. 
\end{proof}

\begin{definition}\label{topFree}
    We say that the action $\action$ is free if $\Aper = X$. We say that the action $\action$ is topologically free if $\Aper$ is dense in $X$, i.e., if $\overline{\Aper} = X$.
\end{definition}

\begin{remark}\label{rmkEmptyInt}
    We can write
    $$\Aper = \bigcap_{s\in G \setminus\{e\}}\Fix{s}^c.$$
    As $X$ is compact and Hausdorff, it is a Baire space. It follows that, as $G$ is countable, we have that $\Aper$ is dense if and only if $\Fix{s}$ has empty interior for every $s\in G \setminus\{e\}$.
\end{remark}

\begin{corollary}\label{conmAndTopFree}
    The action $\action$ is topologically free if and only if $C(X)$ is equal to $C(X)'$.
\end{corollary}
\begin{proof}
    We just note that a nonzero continuous function has support with nonempty interior. Then, the result follows from the above remark and Proposition \refeq{charCommutant}.
\end{proof}

We now prove a couple of technical lemmas that we need in order to prove Theorem \refeq{topFreeThm}.

\begin{lemma}\label{unimFunct}
    Let $x\in X$, and let $k_1,\ldots,k_N \in G$ such that $\sigma_{k_i}(x) \neq x$ for every $i=1,\ldots,N$. Then, there exists $U$ an open neighborhood of $x$ and unimodular functions $\theta_1,\ldots,\theta_{2^N}\in C(X)$ that satisfy the following property: if $f=\sum_{s\in G} f_s \delta_s$ is an element of $\luno$, and we consider
    $$f' = \frac{1}{2^N}\sum_{l=1}^{2^N}\theta_l f \overline{\theta_l},$$
    then, if we write $f' = \sum_{s\in G}f_s' \delta_s$, it verifies that:
    \begin{enumerate}
        \item $f_e' = f_e$,
        \item $f_{k_l}'(y) = 0$ for every $y\in U$ and every $l=1,\ldots,N.$
    \end{enumerate}
\end{lemma}
\begin{proof}
    Let $f=\sum_{s\in G}f_s\delta_s \in \luno$ and let $h\in C(X)$. A direct computation shows that 
    \begin{equation}\label{eqLemma}
        \frac{1}{2}\left( hf\overline{h} + \overline{h}fh\right) = \sum_{s\in G} f_s \text{Re}\left(h\alpha_s \left( \overline{h} \right) \right) \delta_s.
    \end{equation}

    Let us now prove the lemma. We prove it by induction on $N$. Let $N=1$, we have $k_1 \in G$ such that $x\neq \sigma_{k_1}(x)$. We can take an open neighborhood $U$ of $X$ and an unimodular function $\theta_0$ such that $\theta_0$ is equal to $1$ in $U$ and  equal to $-i$ in $\sigma_{k_1}^{-1}(U)$. We set $\theta_1 = \theta_0$ and $\theta_2 = \overline{\theta_0}$. Take $f=\sum_{s\in G}f_s \delta_s\in\luno$. By Equation (\refeq{eqLemma}), we have that
    \begin{align*}
        \begin{split}
            f' = \frac{1}{2} \left( \theta_1 f \overline{\theta_1} + \theta_2 f \overline{\theta_2} \right) = \frac{1}{2} \left( \theta_0 f \overline{\theta_0} + \overline{\theta_0} f \theta_0 \right) = \sum_{s\in G} f_s \text{Re}\left( \theta_0 \alpha_s\left( \overline{ \theta_0 }\right)\right)\delta_s.
        \end{split}
    \end{align*}
    Then, it follows:
     \begin{enumerate}
        \item $f_e' = f_e \text{Re} \left( \theta_0 \alpha_e\left( \overline{\theta_0} \right)\right) = f_e \text{Re} \left( \theta_0 \overline{\theta_0} \right) =  f_e$ because $\theta_0$ is unimodular;
        \item For every $y\in U$: 
        \begin{align*}
            \begin{split}
                f_{k_1}'(y) &= f_{k_1}(y) \text{Re}\left( \theta_0(y) \alpha_{k_1}\left( \overline{\theta_0} (y) \right) \right) = f_{k_1}(y) \text{Re} \left( \theta_0(y) \overline{\theta_0} \left( \sigma_{k_1}^{{-1}}(y) \right) \right)\\ &= f_{k_1}(y) \text{Re}(i) = 0.
            \end{split}
        \end{align*}
    \end{enumerate}

    Now we do the inductive step. Suppose we have $\sigma_{k_1}(x),\ldots, \sigma_{k_{N-1}}(x)$, all distinct from $x$, and we have unimodular functions $\widetilde{\theta}_1, \ldots, \widetilde{\theta}_{2^{N-1}}\in C(X)$ and an open set $\widetilde U$ of $x$ such that if $f=\sum_{s\in G}f_s \delta_s$ and we denote 
    $$\widetilde f = \frac{1}{2^{N-1}} \sum_{l=1}^{2^{N-1}} \widetilde\theta_l f \overline{\widetilde{\theta}_l} = \sum_{s\in G} \widetilde{f}_s \delta_s,$$
    then:
        \begin{enumerate}
        \item $\widetilde{f}_e = f_e$;
        \item $\widetilde{f}_{k_l}(y) = 0$ for every $y\in \widetilde U$ and every $l=1,\ldots,N-1.$
    \end{enumerate}
    
    We have $\sigma_{k_N}(x)\neq x$, so we can take an open neighborhood $V$ of $x$ and a unimodular function $\theta_0$ such that $\theta_0$ is equal to $1$ in $V$ and equal to $-i$ in $\sigma_{k_N}^{-1}(U)$. We set $U=\widetilde{U} \cap V$. We have
    \begin{align*}
        \begin{split}
            \frac{1}{2}\left( \theta_0 \widetilde f \overline{\theta}_0 + \overline{\theta}_0 \widetilde f \theta_0 \right) &= \frac{1}{2}\left( \theta_0 \left[ \frac{1}{2^{N-1}} \sum_{l=1}^{2^{N-1}} \widetilde\theta_l f \overline{\widetilde{\theta}_l} \right] \overline{\theta_0} + \overline{\theta_0}\left[ \frac{1}{2^{N-1}} \sum_{l=1}^{2^{N-1}} \widetilde\theta_l f \overline{\widetilde{\theta}_l} \right] \theta_0 \right)\\
            & = \frac{1}{2^N} \left( \sum_{l=1}^{2^{N-1}} (\theta_0 \widetilde{\theta}_l) f \overline{( \theta_0\widetilde{\theta}_l)} + \sum_{l=1}^{2^{N-1}} (\overline{\theta_0} \widetilde{\theta_l} ) f \overline{(\overline{\theta_0} \widetilde{\theta_l})} \right); \\
        \end{split}
    \end{align*}
    setting $\theta_l = \theta_0\widetilde\theta_l$ for every $l=1,\ldots,2^{N-1}$ and $\theta_l = \overline{\theta_0}\widetilde\theta_{l-2^{N-1}}$ for every $l=2^{N-1}+1,\ldots, 2^N$, we get from the above equation that
    $$f'= \frac{1}{2^N} \sum_{l=1}^{2^N} \theta_l f \overline{\theta_l} = \frac{1}{2}\left( \theta_0 \widetilde f \overline{\theta}_0 + \overline{\theta}_0 \widetilde f \theta_0 \right).$$
    Now, using Equation (\refeq{eqLemma}), we have that
    \begin{align*}
        \begin{split}
            f' = \frac{1}{2}\left( \theta_0 \widetilde f \overline{\theta}_0 + \overline{\theta}_0 \widetilde f \theta_0 \right) = \sum_{s\in G} \widetilde{f_s}\text{Re}\left( \theta_0 \alpha_s\left(  \overline{\theta_0}\right)\right)\delta_s.
        \end{split}
    \end{align*}
    Finally, writing $f' = \sum_{s\in G}f_s' \delta_s$, we obtain:
    \begin{enumerate}
        \item $f_e' = \widetilde{f_e} \text{Re}\left( \theta_0 \overline{\theta_0}\right)= f_e$ as $\theta_0$ is unimodular;
        \item If $y\in U$ and $l=1,\ldots, N-1$, we have that
        \begin{align*}
            \begin{split}
                f_{k_l}'(y) = \widetilde{f}_{k_l}(y) \text{Re} \left( \theta_0(y) \alpha_{k_l} \left( \overline{\theta_0} \right)(y)\right) = 0 
            \end{split}
        \end{align*}
        given that if $y\in U \subset \widetilde{U}$ then $\widetilde{f}_{k_l}(y) =0;$ 
        \item If $y\in U$, as $U\subset V$, we have
        $$f_{k_N}(y) = \widetilde{f}_{k_N}(y) \text{Re} \left( \theta_0 \overline{\theta_0}\left( \sigma_{k_N}^{{-1}}(y) \right)\right) = \widetilde{f}_{k_N}(y) \text{Re}\left( i \right) = 0.$$
    \end{enumerate}
    This ends the proof.
\end{proof}

\begin{lemma}\label{intConm}
    Suppose that $I$ is a closed ideal of $\luno$ such that $I\cap C(X) = \{0\}$. Then, for every $f=
    \sum_{s\in G}f_s\delta_s \in I$, we have that $f_s$ vanishes on $\Aper$ for every $s\in G$.
\end{lemma}
\begin{proof}
    Let $f = \sum_{s\in G} f_s \delta_s \in I$. We want to prove that $f_s(x) = 0$ for every $x\in \Aper$ and for every $s\in G$. It is enough to prove that $f_e(x) = 0$ for every $x\in\Aper$ as, given that $I$ is an ideal, we have that $f \delta_{s^{-1}}$ belongs to $I$ for every $s\in G$. We fix $x\in\Aper$.

    We consider the quotient $\luno / I$. It is a Banach algebra with the following norm:
    $$\lVert h+I \rVert = \inf_{j\in I}\lVert h+j\rVert \text{ for every } h\in\luno.$$
    Let $q:\luno\to\luno/I$ be the canonical projection, which is clearly a contraction. We have that $I \cap C(X) = \{0\}$, so we can define a norm in $C(X)$ using the norm in the quotient: we set $\lVert\>\cdot\>\rVert' = \lVert\>\cdot\>\rVert \circ q $. This norm turns $C(X)$ into a normed algebra, so by \cite[Theorem 1.2.4]{normC(X)} we have that $\lVert\>\cdot\>\rVert_\infty \leq \lVert\>\cdot\>\rVert'$. But, as $q$ is contractive, we also have that $\lVert\>\cdot\>\rVert'\leq \lVert\>\cdot\>\rVert_\infty$, so the two norms must coincide. It follows that the restriction of $q$ to $C(X)$ is an isometry.

    Let $\varepsilon > 0$. We can take a finite subset $F\subset G\setminus\{e\}$ such that
    $$\lVert f-(f_e + \sum_{s\in F} f_s\delta_s )\rVert < \varepsilon.$$
    We set $b$ as $\sum_{s\in F}f_s\delta_s$ and we set $c$ as $f-(f_e + b)$, so we can write $f = f_e + b + c$, with $\lVert c\rVert < \varepsilon$. As $F$ is finite, we have $F = \{k_1\ldots,k_N\}$, and as $x$ belongs to $\Aper$ we have that $\sigma_{k_i}(x) \neq x$ for every $i=1\ldots,N$. Using Lemma \refeq{unimFunct} we obtain unimodular functions $\theta_1,\ldots,\theta_M \in C(X)$ and an open neighborhood $U$ of $x$ such that if we write
    $$f' = \frac{1}{M}\sum_{l=1}^M \theta_l f \overline{\theta_l} = \sum_{s\in G}f'_s\delta_s$$
    then $f_e' = f_e$ and $f_{k_i}'(y) = 0$ for every $y\in U$ and every $i=1,\ldots, N$. It is clear that $f'$ belongs to $I$. We take $\varphi\in C(X)$, with $0\leq \varphi\leq 1$, supported in $U$, and such that $\varphi(x) = 1$. We set $\widetilde f = \varphi f'$, so we have that $\widetilde f$ belongs to $I$. We can write $\widetilde f = \sum_{s\in G} \widetilde f_s \delta_s$, and we have that $\widetilde f_e(x) = \varphi(x) f_e'(x) = f_e(x)$.

    Now, we have the following, where we are using that the functions $\theta_l$ are unimodular in order to get the first summand in the last equality:
    \begin{align*}
        \begin{split}
            \widetilde f &= \varphi f' = \frac{1}{M}\sum_{l=1}^M \varphi\theta_l f \overline{\theta_l}\\
            &=\frac{1}{M}\sum_{l=1}^M\varphi\theta_l f_e \overline{\theta_l} + \frac{1}{M}\sum_{l=1}^M\varphi\theta_l b \overline{\theta_l} + \frac{1}{M}\sum_{l=1}^M\varphi\theta_l c \overline{\theta_l}\\
            &= \varphi f_e + \frac{1}{M}\sum_{l=1}^M\varphi\theta_l b \overline{\theta_l} + \frac{1}{M}\sum_{l=1}^M\varphi\theta_l c \overline{\theta_l}.
        \end{split}
    \end{align*}
    We have that $b=\sum_{s\in F} f_s \delta_s$, so using that $\varphi$ is supported in $U$ and that the functioms $f'_{k_i}$ vanish in $U$ for every $i=1\ldots,N$, we obtain
    $$\frac{1}{M}\sum_{l=1}^M \varphi \theta_l b \overline{\theta_l} = \sum_{s\in F} \varphi f_s'\delta_s = \sum_{i=1}^N \varphi f_{k_i}'\delta_{k_i} = 0.$$
    We have then that
    $\widetilde f = \varphi f_e + \frac{1}{M}\sum_{l=1}^M\varphi\theta_l c \overline{\theta_l}.$
    If we set $\widetilde c = \frac{1}{M}\sum_{l=1}^M \varphi \theta_l c \overline{\theta_l}$, we compute
    $$\lVert \widetilde c \rVert = \lVert \frac{\varphi}{M} \sum_{l=1}^M \theta_l c \overline{\theta_l} \rVert \leq \frac{1}{M} \sum_{l=1}^M \lVert \theta_l \rVert \lVert c \rVert \lVert \overline{\theta_l} \rVert = \frac{M}{M} \lVert c \rVert < \varepsilon.$$ 
    Now, as $\widetilde f$ is in $I$, we have $q(\widetilde f) = 0$. This implies that $q(\varphi f_e) = - q(\widetilde c)$. Then, as $q$ is a contraction and its restriction to $C(X)$ is an isometry, we obtain that
    \begin{align*}
        \begin{split}
            |f_e(x)| = |\varphi(x) f_e(x)| \leq \lVert \varphi f_e \rVert_\infty = \lVert q(\varphi f_e)\rVert = \lVert q(\widetilde c) \rVert \leq \lVert \widetilde c\rVert < \varepsilon.
        \end{split}
    \end{align*}
    As $\varepsilon$ was arbitrary, we conclude that $f_e(x)=0$.
\end{proof}

Finally, we arrive at the main theorem of this section. We point that equivalence between (1) and (3) in the theorem can be seen as an interpretation for $\luno$ of the more general context \cite[Theorem 5.9]{twistedCrossProd}, but the proof we include here is following \cite[Theorem 4.1]{1}, of which this result is a generalization to arbitrary countable (infinite) discrete groups.
In order to prove this result we make use of the family of representations $\{\pi_x : \luno \to B(H_x), x\in X\}$ that we introduced in the preliminaries (see Section \ref{secRep}).

\begin{theorem}\label{topFreeThm}
    The following are equivalent:
    \begin{enumerate}[(1)]
        \item For every closed nonzero ideal $I$ of $\luno$ the intersection $I\cap C(X)$ is nonzero.
        \item For every closed and self-adjoint nonzero ideal of $\luno$ the intersection $I\cap C(X)$ is nonzero.
        \item The action $\action$ is topologically free.
    \end{enumerate}
\end{theorem}
\begin{proof}
    It is clear that (1) implies (2). We prove first that (2) implies (3). Suppose that the action $\action$ is not topologically free. Then, by Remark \refeq{rmkEmptyInt}, there exists some $s_0 \in G\setminus\{e\}$ such that $\Fix{s_0}$ has nonempty interior. We can then take a nonzero $f\in C(X)$ such that $\supp{f}$ is contained in $\Fix{s_0}$. We define $I$ as the closed and self-adjoint ideal of $\luno$ generated by $f-f\delta_{s_0}$. 

    Let $x\in X$, and we consider $\pi_x : \luno \to B(H_x)$ the representation associated to $x$. We remember that $H_x$ is the Hilbert space with orthonormal basis $\{v_\beta : \beta \in J_x\}$, where $\{s_\beta: \beta\in J_x\}$ is a set of representatives of the quotient $G/G_x =\{s_\beta G_x : \beta \in J_x\}$. We will prove that $\pi_x$ vanishes at $f-f\delta_{s_0},$ which implies that $\pi_x$ vanishes on $I$. Let $\beta \in J_x$, we have that there are unique $\gamma \in J_x$ and $s'\in G_x $ such that $s_0 s_\beta = s_\gamma s' \in s_\gamma G_x$. Then:
    \begin{align*}
        \begin{split}
            \pi_x(f-f\delta_{s_0}) v_\beta &= \pi_x(f)v_\beta - \pi_x(f)\pi_x(\delta_{s_0})v_\beta
            = f(\sigma_{s_\beta} ( x)) v_\beta - \pi_x(f) v_\gamma\\
            &= f(\sigma_{s_\beta} (x))v_\beta - f(\sigma_{s_\gamma} (x))v_\gamma.
        \end{split}
    \end{align*}
    Suppose first that $\sigma_{s_\beta}(x)$ is in $\Fix{s_0}$. Then, we have that 
    \begin{align*}
        \begin{split}
            \sigma_{s_\beta}  (x) = \sigma_{s_0}(  \sigma_{s_\beta}(x)) = \sigma_{s_o s_\beta} (x) = \sigma_{s_\gamma} (\sigma_{s'}(x)) = \sigma_{s_\gamma} (x),
        \end{split}
    \end{align*}
    and as $\{s_\beta: \beta \in J_x\}$ are representatives of $G/G_x$, it must be that $\beta$ equals $\gamma$. Then, it's clear that $\pi_x(f-f\delta_{s_0}) v_\beta = 0$. Suppose now that $\sigma_{s_\beta}(x)$ is not in $\Fix{s_0}$. As $f$ is supported on $\Fix{s_0}$, we have that $f(\sigma_{s_\beta}(x)) = 0$. We can write $s_\gamma = s_0 s_\beta (s')^{-1}$. Then
    $$f(\sigma_{s_\gamma} (x)) = f(\sigma_{s_0} (\sigma_{s_\beta} (\sigma_{s'}^{-1}(x)))) = f(\sigma_{s_0}(\sigma_{s_\beta}  (x))).$$
    As $\Fix{s_0}$ is invariant under $\sigma_{s_0}$, we have that $\sigma_{s_0}(\sigma_\beta(x))$ is not in $\Fix{s_0}$ and, again, as $f$ is supported on $\Fix{s_0}$, it must be that $f(\sigma_{s_0}(\sigma_{s_\beta}(x))) = 0.$ It follows that $\pi_x(f-f\delta_{s_0})v_\beta = 0$ in both cases. As this stands true for every $\beta \in J_x$, we have that $\pi_x$ vanishes at $f-f\delta_{s_0}$, and as a consequence it also vanishes on $I$.

    Let $g\in I\cap C(X)$. Then, for every $x\in X$ we have that $0=\pi_x(g)v_e = g(x) v_e$, so $g(x) =0$. This means that $g=0$, so $I$ is a closed and self-adjoint nonzero ideal such that $I\cap C(X)=\{0\}$. This contradicts (2). 

    Now, let's prove that (3) implies (1). Let $I$ be a closed ideal of $\luno$ such that $I\cap C(X) = \{0\}$. Let $f = \sum_{s\in G}f_s\delta_s\in I$. Because of Lemma \refeq{intConm} we have that $f_s$ vanishes on $\Aper$, for every $s\in G$. As $\action$ is topologically free we have that $\Aper$ is dense in $X$, it follows that $f_s = 0$ for every $s\in G$, so $f=0$. This implies that $I=\{0\}$.
\end{proof}

Theorem \refeq{topFreeThm} allows us to relatively easily characterize dynamical properties of $\action$ as analytic-algebraic properties of $\luno$, and vice versa. We do so in the rest of this section. The proofs of many of these results are heavily inspired by the ones in \cite{1} as, given that we have Theorem \refeq{topFreeThm} at our disposal, the arguments are particularly similar now.

\begin{definition}\label{defMinimality}
    We say that the action $\action$ is minimal if there are no proper closed subsets of $X$ that are invariant under the action of $G$. This is equivalent to the orbit of every element being dense in $X$.
\end{definition}

Our next objective is to characterize the action being minimal as an analytic-algebraic property of $\luno$. In order to do so, we have to understand how $G$-invariant and closed subsets of $X$ give rise to closed (and self-adjoint) ideals of $\luno$.

Let $D$ be a closed subset of $X$ invariant under the action of $G$. We can define the following:
$$I_D = \left\{ f=\sum_{s\in G} f_s\delta_s \in \luno : f_s|_D = 0 \text{ for every } s\in G\right\}.$$
It's straightforward to check that $I_D$ is a closed and self-adjoint ideal of $\luno$. The following propositions tell that we can go the other way around, that is, to obtain a closed $G$-invariant subset of $X$ from a closed ideal of $\luno$.

\begin{proposition}\label{propIntId}
    Let $I$ be a closed ideal of $\luno$. Then, there exists a closed and $G$-invariant subset $D$ of $X$ such that
    $$I\cap C(X) = \{f\in C(X) : f|_D = 0\}.$$
\end{proposition}
\begin{proof}
    As $I$ is a closed ideal of $\luno$ it is clear that $I\cap C(X)$ is a closed ideal of $C(X)$. It is well known that every closed ideal of $\luno$ consists of the continuous functions that vanish on some closed subset of $X$. This means that there is some closed subset $D$ of $X$ such that 
    $$I\cap C(X) = \left\{ f\in C(X) : f|_D =0\right\}.$$
    We just have to check that $D$ is invariant under the action of $G$. Suppose that there is some $s\in G$ such that $\sigma_s(D)$ is not included in $D$. This means that there is some $x\in D$ such that $\sigma_s (x) \not\in D$. By Urysohn's lemma we can take $g\in C(X)$ such that $g$ vanishes on $D$ and $g(\sigma_s(x)) \neq 0$. As $I$ is an ideal we have that $(\delta_s)^{-1}g \delta_s = g\circ \sigma_s\in I$, so in particular it belongs to $I\cap C(X)$. It follows that $g(x)=0$, which is a contradiction.
\end{proof}

The following theorem is a generalization of \cite[Theorem 4.2]{1}, where this was proven for $\Z$-actions. It characterizes simplicity of the algebra as minimality of the action. Equivalence between (1) and (3) can be seen as a consequence of the more general \cite[Theorem 6.13]{twistedCrossProd}.

\begin{theorem}\label{thmMinSimp}
    Suppose that the action $\action$ is topologically free. Then, the following are equivalent:
    \begin{enumerate}[(1)]
        \item There are no proper closed ideals in $\luno$.
        \item There are no proper closed and self-adjoint ideals in $\luno$.
        \item The action $\action$ is minimal.
    \end{enumerate}
\end{theorem}
\begin{proof}
    That (1) implies (2) is evident. Let's prove that (2) implies (3). Suppose that the action $\action$ is not minimal, then there is some $x\in X$ such that its orbit is not dense in $X.$ Let $D$ be the closure of the orbit of $x$. Then $D$ is a closed and $G$-invariant proper subset of $X$, so
    $$I_D = \left\{ f=\sum_{s\in G} f_s\delta_s \in \luno : f_s|_D = 0 \text{ for every } s\in G\right\}$$
    is a proper closed and self-adjoint ideal of $\luno.$ This contradicts (2).

    We now prove that (3) implies (1). Let $I$ be a closed nonzero ideal of $\luno$. As the action $\action$ is topologically free, Theorem \refeq{topFreeThm} tells us that $I\cap C(X)$ is nonzero. So, using Proposition \refeq{propIntId} we get a closed $G$-invariant subset $D$ that is not the whole space. But as $\action$ is minimal, the only $G$-invariant and closed subsets of $X$ are itself and the empty set. Then, it must be that $D$ is the empty set, so $I\cap C(X) = C(X)$. This implies that $I = \luno$. 
\end{proof}

\begin{definition}\label{defTransitivity}
    We say that the action $\action$ is topologically transitive if for every pair of nonempty open subsets $U$ and $V$ of $X$ there is some $s\in G$ such that $\sigma_s(U)\cap V$ is nonempty.
\end{definition}

We characterize when the action is topologically transitive in terms of the properties of the algebra. We need the following lemma.

\begin{lemma}\label{lemmaTransitive}
    The following are equivalent:
    \begin{enumerate}
        \item There exist two nonempty disjoint and $G$-invariant open subsets $O_1$ and $O_2$ of $X$ such that $X = \overline{O_1} \cup \overline{O_2}$.
        \item The action $\action$ is not topologically transitive.
    \end{enumerate}
\end{lemma}
\begin{proof}
    That (1) implies (2) is trivial. Let's prove that (2) implies (1). As $\action$ is not topologically transitive, there are two nonempty and open subsets $U$ and $V$ of $X$ such that $\sigma_s(U) \cap V$ is empty for every $s\in G$. Set
    $O_1 = \bigcup_{s\in G}\sigma_s(U).$ It's clear that $O_1$ is an open $G$-invariant subset of $X$. As $V$ is contained in $\overline{O_1}^c$, we can take $O_2 = \overline{O_1}^c$. These sets verify what we need.
\end{proof}

The following theorem generalizes \cite[Theorem 4.10]{1}, which was proved for $\Z$-actions. It characterizes some notion of primality of the algebra as the action being topologically transitive. In the proof, given a closed subset $D$ of $X$, we denote by $\ker(D)$ the set $\left\{ g\in C(X) : g|_D = 0 \right\}.$ We recall that $E$ denotes the the canonical projection from $\luno$ to $C(X)$ given in Definition \refeq{projE}.

\begin{theorem}\label{thmTransitivity}
    Suppose that the action $\action$ is topologically free. Then, the following are equivalent:
    \begin{enumerate}
        \item If $I$ and $J$ are closed nonzero ideals of $\luno$ then $I\cap J$ is nonzero.
        \item If $I$ and $J$ are closed and self-adjoint nonzero ideals of $\luno$ then $I\cap J$ is nonzero.
        \item The action $\action$ is topologically transitive.
    \end{enumerate}
\end{theorem}
\begin{proof}
    It is evident that (1) implies (2). Let's see that (2) implies (3). Suppose that the action $\action$ is not topologically transitive, then by Lemma \refeq{lemmaTransitive} we have two nonempty open sets $O_1$ and $O_2$, disjoint and $G$-invariant, such that $X = \overline{O_1} \cup \overline{O_2}$. We define 
    $$I_i = \left\{ \sum_{s\in G}f_s\delta_s: f_s \in \ker(\overline{O_i}) \text{ for every } s\in G \right\}$$
    for $i=1,2$. Both are closed and self-adjoint ideals of $\luno$. Then
    $$E(I_1 \cap I_2) \subset E(I_1) \cap E(I_2)  = \ker(\overline{O_1}) \cap \ker(\overline{O_2}) = \ker(\overline{O_1} \cup \overline{O_2}) = \ker(X) = \{0\},$$
    so $I_1 \cap I_2 = \{0\}$. This contradicts (2).

    Let's now prove that (3) implies (1). Let $I_1$ and $I_2$ be two proper and closed ideals of $\luno$ such that $I_1\cap I_2 = \{0\}$. As $\action$ is topologically free, Theorem \refeq{topFreeThm} tells us that $I_1 \cap C(X)$ and $I_2\cap C(X)$ are both proper ideals of $C(X)$, so by Proposistion \refeq{propIntId} we get two proper closed and $G$-invariant subsets $D_1$ and $D_2$ of $X$ such that $I_1\cap C(X) = \ker(D_1)$ and $I_2 \cap C(X) = \ker(D_2)$. As $I_1 \cap I_2 = \{0\}$ we also have that $I_1\cap I_2 \cap C(X) =\{0\}$. Then
    $$\{0\} = I_1 \cap I_2 \cap C(X) = \ker(D_1) \cap \ker(D_1) = \ker(D_1 \cup D_2),$$
    so it must be $D_1 \cup D_2 = X$. As both are proper and closed, they must both have nonempty interior, so if we denote by $U$ the interior of $D_1$ and by $V$ the complement of $D_1$ we have that $\sigma_s(U) \cap V = \emptyset$ for every $s\in G$. This contradicts the topological transitivity of $\action$.
\end{proof}

In Proposition \refeq{propIntId} we mentioned a way of relating closed ideals of $\luno$ to closed ideals of $C(X)$: if $I$ is a closed ideal of $\luno$, then $I\cap C(X)$ is a closed ideal of $C(X)$. A reasonable question to ask is whether $C(X)$ detects all closed ideals of $\luno$ through this procedure, i.e., if $I$ is a closed nonzero ideal of $\luno$, then $I\cap C(X)$ is a nonzero ideal of $C(X)$. This property is a case of the general property called the \textit{ideal intersection property}. This property has been studied in different contexts, see, for example, \cite{isp1,isp2}, where the property was studied between different (reduced) $C^*$-algebras and subalgebras of them, and \cite[Theorem 5.6]{fellBundlesISP}, where it was studied for the algebra of crossed sections associated to a Fell bundle over a group with polynomial growth and its universal $C^*$-algebra completion. In our context, Theorem \refeq{topFreeThm} tells us exactly when this holds between $\luno$ and $C(X)$, that is, when the action is topologically free.

Another similar property is the \textit{ideal separation property}, which was also studied in \cite{isp2}. This property holds if, for any pair of closed ideals of $\luno$, they are equal if they have the same intersection with $C(X)$. Explicitly, $I\cap C(X) = J \cap C(X)$ implies $I=J$ for every pair of closed ideals $I$ and $J$ of $\luno$. So, this property holds if we can distinguish closed ideals of $\luno$ through their intersection with $C(X)$, and Theorem \refeq{topFreeThm} tells us that for this to hold we need at least the action to be topologically free.  It turns out that the ideal separation property is related to the definition that follows, which is a strengthening of topological freeness.

\begin{definition}\label{resTopFree}
    We say that the action $\action$ is residually topologically free if for any $G$-invariant closed subset $D$ of $X$, the restriction of the action of $G$ to $D$ is topologically free. For brevity, we sometimes write r.t.f. in place of residually topologically free.
\end{definition}

\begin{remark}\label{rtfAbelian}
    For abelian groups, the notions of free actions and of r.t.f. actions are equivalent. It is evident that if the action is free then it is also r.t.f. For the other direction, suppose that there is some $x\in X$ such that $G_x \neq \{e\}$. Take $D$ as the closure of the orbit of $x$. As $G$ is abelian, all points of the orbit of $x$ have the same stabilizer subgroup $G_x$. Then, given $t\in G_x$, for every $s\in G$ we have that $\sigma_t(\sigma_s(x)) = \sigma_s(\sigma_t(x)) = \sigma_s(x).$ This means that $t$ fixes the whole orbit of $x$, so it also fixes $D$, so the restriction of the action to $D$ is not topologically free.
\end{remark}

\begin{example}
    Let $G$ be the free group on two generators $\mathbb{F}_2$, acting canonically on its Gromov boundary $X=\partial \mathbb{F}_2$. This action is topologically free but not free. The action is also minimal, so it has no proper $G$-invariants subsets, so this action is residually topologically free but not free.
\end{example}

The main result of the rest of this section is Theorem \refeq{rtfThm}. We note that this theorem can be seen as a consequence of the more general \cite[Theorem 6.12]{twistedCrossProd}. We present here a self-contained proof for this context, where the ideas draw inspiration from \cite[Section 1]{resTopFree}, where something similar is done for $C^*$-dynamical systems.

Let us fix some notation. Suppose that $D$ is a closed $G$-invariant subset of $X$. Then, we can restrict to $D$ the action of $G$ and obtain a new action, that we denote by $G\curvearrowright D$. As $G\curvearrowright D$ is an action, we can consider the algebra $\ell^1(G\curvearrowright D),$ and the map $R_D: \luno \to \ell^1(G \curvearrowright D)$ given by
$$R_D\left( \sum_{s\in G}f_s \delta_s\right) = \sum_{s\in G} f_s|_{D} \delta_s \text{ for every } \sum_{s\in G}f_s \delta_s \in \luno;$$
here, $\delta_s$ also denotes the canonical unitary element of $\ell^1(G\curvearrowright D)$ associated to each $s\in G$. It is clear that $R_D$ is a surjective and norm-one *-homomorphism, and that $\ker(R_D) = \{\sum_{s\in G} f_s \delta_s : f_s|_D = 0 \text{ for every } s\in G\}$.

\begin{lemma}\label{lemmaRTF}
    Let $I$ be a closed ideal of $\luno$, and let $D$ be a closed $G$-invariant subset of $X$ such that $I \cap C(X) = \{f\in C(X) : f|_D=0\}$ (as in Proposition \refeq{propIntId}). Then $R_D(I)$ has zero intersection with $C(D)$.
\end{lemma}
\begin{proof}
    Let $g\in R_D(I) \cap C(D)$, which means that there is some $\sum_{s\in G} f_s \delta_s\in I$ that gets mapped to $g$ through $R_D.$ We have then that
    $$\sum_{s\in G}f_s|_{D}\delta_s = g,$$
    so it follows that $f_s|_D = 0$ for every $s\in G\setminus\{e\}$ and $f_e|_D = g$. This implies that $\sum_{s\neq e} f_s \delta_s \in \ker(R_D)$. As $\{f \in C(X) : f|_D = 0\}$ is equal to $I\cap C(X)$, which in turn is contained in $I$, it's clear that $\ker (R_D) \subset I$, so it follows that
    $$f_e = \sum_{s\in G} f_s\delta_s - \sum_{s\neq e}f_s\delta_s \in I.$$
    Then $f_e \in I \cap C(X)$, so its restriction to $D$ is equal to the zero function. As $g$ is equal to $f_e|_D$, we conclude that $ R_D(I) \cap C(D) = \{0\}$.
\end{proof}

\begin{remark}\label{rmkResTopFree}
    Given a bounded and surjective linear operator $T:A\to B$ between Banach spaces, and given $I$ a nonempty subset of $A$, we have that $T(I)$ is closed in $B$ if and only if $I + \ker(T)$ is closed in $A$. We use this in the following theorem.
\end{remark}

In the next proof, given $D$ a closed $G$-invariant subset of $X$, we denote by $r_D:C(X) \to C(D)$ the map given by $r_D(f) = f|_D$ for every $f\in C(X)$. We also denote by $E_D : \ell^1(G \curvearrowright D) \to C(D)$ the map given by $E_D(\sum_{s\in G} f_s\delta_s) = f_e$ for every $\sum_{s\in G} f_s\delta_s\in\ell^1(G\curvearrowright D)$. It is clear that $E_D\circ R_D =r_D \circ E$.

\begin{theorem}\label{rtfThm}
    The following are equivalent:
    \begin{enumerate}
        \item For every pair of closed ideals $I$ and $J$ of $\luno$, if $I \cap C(X)$ is equal to $J \cap C(X)$ then $I=J$.
        \item $E(I)$ is contained in $I$ for every closed ideal $I$ of $\luno$.
        \item The action $\action$ is residually topologically free.
    \end{enumerate}
\end{theorem}
\begin{proof}
    Let's first prove that (1) implies (2). Let $I$ be a closed ideal of $\luno$. Using Proposition \refeq{propIntId} we have that there is some closed and $G$-invariant subset $D$ of $X$ such that $I\cap C(X)$ consists of the continuous functions that vanish on $D$. We define the closed ideal
    $$J=\left\{ \sum_{s\in G} f_s \delta_s : f_s|_D = 0 \text{ for every } s \in G\right\}.$$
    It is clear that $J\cap C(X) = I\cap C(X)$, so using (1) we have that $J = I$. Now it follows that $E(I) = I\cap C(X) \subset I$.

    We prove now that (2) implies (3). Let $D$ be a closed and $G$-invariant subset of $X$, and consider the restriction of the action to $D$. In order to prove that $G\curvearrowright D$ is topologically free, we use Theorem \refeq{topFreeThm}. Let $J$ be a closed ideal of $\ell^1(G\curvearrowright D)$ such that $J\cap C(D) = \{0\}$, we have to prove that $J=\{0\}$. We define $I = R_D^{-1}(J)$, which is a closed ideal of $\luno$. By (2) we have that $E(I) \subset I$, so in particular $E(I) \subset I \cap C(X)$. Now, we have that
    $$E_D(J) = E_D(R_D(I)) = r_D(E(I)) \subset r_D(I\cap C(X)) \subset J \cap C(D) = \{0\}.$$
    As $E_D(J) =\{0\}$ it follows that $J=\{0\}$.

    That (3) follows from (1) is the only implication left to prove. Suppose that $I$ and $J$ are two closed ideals of $\luno$ such that $I\cap C(X) = J\cap C(X)$. Using Proposition \refeq{propIntId} we get a closed and $G$-invariant subset $D$ of $X$ such that both $I\cap C(X)$ and $J \cap C(X)$ are equal to $\{f\in C(X): f|_D = 0\}.$ We have that 
    $$\ker(R_D) = \left\{\sum_{s\in G}f_s \delta_s: f_s|_D = 0 \text{ for every }s \in G\right\}$$
    is included in $I$ and $J$. As $R_D$ is contractive and surjective, and $I$ and $J$ are closed, we have by Remark \refeq{rmkResTopFree} that both $R_D(I)$ and $R_D(J)$ are closed ideals of $\ell^1(G\curvearrowright D)$. Now, using Lemma \refeq{lemmaRTF} we get that $R_D(I)\cap C(D)$ and $R_D(I)\cap C(D)$ are both equal to $\{0\}$. By (3), we have that $G\curvearrowright D$ is topologically free, so Theorem \refeq{topFreeThm} implies that $R_D(I) = R_D(J) = 0$. This means that $I$ and $J$ are included in $\ker (R_D)$. As we have proved both inclusions hold, we conclude that $I = \ker (R_D) = J$.
\end{proof}
\begin{remark}
    The theorem above bears a resemblance to some of the work done in \cite[Section 3]{3}. There, for $\Z$-actions, two notions of noncommutative \textit{spectral synthesis} are studied for $\luno$. It is proved there that (1) and (2) of the above theorem are equivalent to the action being free, which makes sense given that $\Z$ is abelian and Remark \refeq{rtfAbelian}. 
\end{remark}

\begin{remark}
    Theorem \refeq{rtfThm} implies that, when the action $\action$ is r.t.f., then we can completely describe the (closed) ideal structure of $\luno$. Concretely, we have that every ideal closed $I$ is such that there exists a $G$-invariant closed subset $D$ of $X$ that satisfies
    $$I=\left\{ \sum_{s\in G}f_s\delta_s : f_s|_D = 0 \text{ for every } s\in G \right\}.$$
    In this situation, we refer to $I$ as the ideal generated by $D$.
    It follows that, when the action is residually topologically free, there is a bijection between the closed ideals of $\luno$ and the $G$-invariant closed subsets of $X$.
\end{remark}

    An ideal $I$ is prime if it is not the whole algebra and $JK \subset I$ implies $J\subset I$ or $K\subset I$, for every pair of ideals $J$ and $K$. An ideal $I$ is called semiprime if $J^2 \subset I$ implies $J\subset I$. Equivalently, an ideal is semiprime if it is an intersection of prime ideals (see \cite[Theorem 10.11]{nonCommRings}). In $C^*$-algebras, every closed ideal is semiprime, given that every closed ideal is the intersection of primitive ideals and primitive ideals are prime. Recently, there has been an interest on the study of non-closed prime and semiprime ideals in $C^*$-algebras, see \cite{primes1,primes2,primes3}. With the above remark at our disposal we are able to show that, if the action $\action$ is r.t.f., then every closed ideal of $\luno$ is semiprime. We can also characterize closed prime ideals of $\luno$. 

    We recall that a minimal set $C$ of the action $G\curvearrowright X$ is a closed and $G$-invariant subset of $X$ such that it does not contain any proper closed and $G$-invariant subset. This is equivalent to the restriction of the action to $C$ being minimal.

\begin{corollary}
    Let $\action$ be residually topologically free. Then, every closed ideal of $\luno$ is semiprime. Furthermore, a closed ideal $I$ of $\luno$ is prime if and only if there exists some closed and minimal $G$-invariant subset $C$ of $X$ such that
    $$I = \left\{ \sum_{s\in G}f_s\delta_s \in \luno : f_s|_C = 0 \text{ for every } s\in G \right\}.$$
\end{corollary}
\begin{proof}
    In order to check that a closed ideal $I$ is prime, it is enough to check that the condition for primeness holds with respect to closed ideals. That is, if $J$ and $K$ are closed ideals such that $JK \subset I$, then $J\subset I$ or $K\subset I$. This follows from the fact that the product ideal $\overline{J} \, \overline{K}$ is contained in $\overline{JK}$ and that, if $I$ is closed, then $\overline{I} = I$. The same argument justifies that, in order to check that a closed ideal $I$ is semiprime, it is enough to prove that if $J$ is a closed ideal such that $J^2$ is contained in $I$ then $J$ is included in $I$.
    
    Let $I$ be a closed ideal and, as the action is r.t.f., let $C$ be the closed $G$-invariant subset of $X$ such that $I$ is generated by $C$, i.e.:
    $$I=\left\{ \sum_{s\in G}f_s\delta_s : f_s|_C = 0 \text{ for every } s\in G\right\}.$$
    For every closed $G$-invariant subset $D$ of $X$, we denote by $J_D$ the closed ideal of $\luno$ generated by $D$, which is defined analogously to $I$. It is clear that, for every pair of closed $G$-invariant sets $E$ and $D$, if $C$ is contained in $E\cup D$ then $J_E J_D$ is contained in $I$. The reciprocal also holds, that is, if $J_E J_D$ is contained on $I$, then $C$ is contained in $E\cup D$. This follows from the fact that if $J_E J_D$ is contained in $I$, then for every $f$ and $g$ belonging to $C(X)$ and vanishing on $E$ and $D$, respectively, the product $fg$ belongs to $I$, so it must vanish on $C$. This happens only if $C$ is contained in $D\cup E$. So, we have that $J_D J_E$ is contained in $I$ if and only if $C$ is contained in $D\cup E$.

    Let us observe that $J_D$ is contained in $I$, for some closed $G$-invariant set $D$, if and only if $C$ is contained in $D$. Now, as the action is residually topologically free, every closed ideal of $\luno$ is of the form $J_D$ for some closed $G$-invariant set $D$. This means that we can easily translate the ideal $I$ being semiprime or prime to a condition on the closed $G$-invariants sets of the action. First, $I$ is always semiprime because $J_D^2$ being contained in $I$ is equivalent to $C$ being contained in $D$, which implies that $J_D$ is contained in $I$. Finally, the ideal $I$ is prime if and only if for every pair $D$ and $E$ of closed $G$-invariants subsets of $X$ such that $C$ is contained in $D\cup E$ we have that $C$ is contained in $D$ or $C$ is contained in $E$. This is equivalent to $C$ being a minimal set for the action.    
\end{proof}

\section{Detecting freeness through self-adjoint ideals}
We mentioned in the introduction the result \cite[Theorem 4.4]{1}. This theorem states that, for $\Z$-actions, the action is free if and only if every closed ideal of $\lone{\Z}{X}$ is self-adjoint. This theorem is interesting in the sense that freeness of the action can be detected by $\lone{\Z}{X}$ through the property of every closed ideal being automatically self-adjoint. This is not something that can be detected by its $C^*$-envelope, as in a $C^*$-algebra all closed ideals are self-adjoint. We would like to generalize this result to this setting, but some problems arise when trying to do this. First of all, the following example shows that \cite[Theorem 4.4]{1} is not true in general.

\begin{example}\label{ctrExample}
We have an action $\varphi : \Z_2 \to \text{Aut}(\Z)$ of $\Z_2$ on $\Z$, given by $\varphi(0) = \text{Id}$ and $\varphi (1) = -\text{Id}$, so we can consider the semidirect product $G = \Z \rtimes_\varphi \Z_2$.  Let $\theta$ be an irrational number. Now, let $X$ be the circle $S^1 = \R / \Z$, and let $G$ act on $X$ the following way: $\sigma_{(n,0)}(\beta) = \beta + n\theta$ (modulo $\Z$) and $\sigma_{(0,1)}(\beta) = -\beta$ for every $\beta \in \R  /  \Z$. It is straightforward to check that this defines an action on $X$, given by 
$$\sigma_{(n,m)}(\beta) = (-1)^m\beta +n\theta.$$
If we identify $\R / \Z$ with the unit circle on the complex plane, what we have is that $(1,0)$ acts as an irrational rotation of angle $\theta$ and that $(0,1)$ acts as a symmetry with respect to the real axis. As $\theta$ is irrational we have that the action $\action$ is minimal. It is easy to verify that $\Per = \left\{ \beta = \frac{n\theta + m}{2}\in\R/\Z : n,m\in \Z\right\}$, so the action is topologically free but not free. But as the action is topologically free and minimal, by Theorem \refeq{thmMinSimp} we have that $\luno$ has no proper closed ideals so, trivially, every closed ideal of $\luno$ is self-adjoint.
\end{example}

We remark that the proof of one of the implications of \cite[Theorem 4.4]{1} generalizes well to this setting: if the action $\action$ is free, then every closed ideal of $\luno$ is self-adjoint. It is interesting that, in that proof, what is actually used is that if the action is free then it is also r.t.f, and that implies that every closed ideal is self-adjoint. This is not a big deal if we are working with $\Z$-actions because Remark \refeq{rtfAbelian} tells us that freeness and residual topological freeness are equivalent for abelian groups. But if we are working with a nonabelian group, we get that being r.t.f, which is weaker than the action being free, is enough to have that every closed ideal is self-adjoint. Regardless, we are able to prove this implication easily as a simple corollary of Theorem \refeq{rtfThm}.

\begin{corollary}\label{libImpSelfAdj}
    If the action $\action$ is residually topologically free, then every closed ideal of $\luno$ is self-adjoint.
\end{corollary}
\begin{proof}
    Let $I$ be a closed ideal. As $C(X)$ is a $C^*$-algebra every closed ideal of it is self-adjoint, so it must be that $I \cap C(X) = I^* \cap C(X)$. As $\action$ is r.t.f, by Theorem \refeq{rtfThm} we obtain that $I=I^*$.
\end{proof}
As we said before, freeness of the action is stronger than it being r.t.f, so the above theorem also works if we change one condition by the other. We now try to generalize the reciprocal of the theorem. We cannot do the same that was done in \cite[Theorem 4.4]{1} because there it is used that if the action of $\Z$ is not free then there is a finite orbit, which is not necessarily true for arbitrary discrete and countable groups.

We denote by $\ker(\alpha)$ the set of elements of $G$ that act trivially on $X$. In what follows, given a discrete group $H$, we denote by $\ell^1(H)$ the Banach *-algebra of integrable functions with respect to the counting measure, with the usual convolution and involution, and we denote its norm by $\lVert\>\cdot\>\rVert_1$. Given $t\in H$ we denote by $\mathbbm{1}_{t}$ the characteristic function of $\{t\}$, which belongs to $\ell^1(H)$, and we recall that $(\mathbbm{1}_{t} * \xi)(h) = \xi (t^{-1}h)$ for every $\xi \in \ell^1(H)$ and every $h\in H$, where $*$ denotes the usual convolution in $\ell^1(H)$.

\begin{lemma}\label{propLiftIdeal}
    Let $G$ be abelian. Suppose that $H$ is a subgroup of $G$ contained on $\ker(\alpha)$ and such that $\ell^1(H)$ has a closed and non-self-adjoint ideal. Then, there is a closed and non-self-adjoint ideal in $\luno$.
\end{lemma}
\begin{proof}
    Let $J$ be a closed and non-self-adjoint ideal of $\ell^1(H)$. For every $x\in X$, every $s\in G$ and every $f \in\luno$, we define $\phixs{x}{s}{f}$ as the map from $H$ to $\C$ given by $\phixs{x}{s}{f} (h) = f(sh)(x)$ for every $h\in H$. We have that
    \begin{equation*}
        \lVert \phixs{x}{s}{f}\rVert_1  = \sum_{t\in H} |f(st)(x)| \leq \sum_{t\in H} \lVert f(st) \rVert_\infty \leq \sum_{s\in G} \lVert f(s) \rVert_\infty = \lVert f \rVert < \infty,
    \end{equation*}
    so $\phixs{x}{s}{f}$ belongs to $\ell^1(H)$ for every $x\in X$, every $s\in G$ and every $f\in\luno$. Then, we can think of $\phi_{x,s}$ as a contracting linear map from $\luno$ to $\ell^1(H)$. We define
    $\widetilde J = \left\{ f \in \luno : \phixs{x}{s}{f} \in J \text{ for every } x\in X \text{ and } s\in G\right\}.$
    We will prove that $\widetilde J$ is a closed and non-self-adjoint ideal of $\luno$. 

    We can write 
    $$\widetilde J = \bigcap_{\substack{x\in X\\ s\in G}} \phi_{x,s}^{-1}(J)$$
    and, as for every $x\in X$ and every $s\in G$ we have that $\phi_{x,s}$ is contractive and $J$ is closed, we have that $\widetilde J$ is also closed. We have to check that it is an ideal. We write $G/H=\{s_\beta H: \beta \in I_H\}$ where $\{s_\beta : \beta \in I_H\}$ is a set of representatives of the quotient, so we have that $G = \bigsqcup_{\beta\in I_H} s_\beta H$. Let $f\in\luno$ and let $g\in\widetilde J$, and fix $x\in X$ and $s\in G$. Let $h\in H$, we have that
    \begin{align*}
        \begin{split}
            \phixs{x}{s}{fg}(h) &= (fg)(sh)(x) = \left[ \sum_{i\in G} f(i) \alpha_i( g (i^{-1}sh))\right](x)\\
            &= \left[ \sum_{\beta \in I_H} \sum_{t\in H} f(s_\beta t) \alpha_{s_\beta t} (g((s_\beta t)^{-1} sh))\right](x).
        \end{split}
    \end{align*}
    Let's fix $\beta\in I_H$. We have that $H \subset \ker(\alpha)$ and that $G$ is abelian, so
    \begin{align*}
        \begin{split}
                \left[ \sum_{t\in H} f(s_\beta t) \alpha_{s_\beta t} (g((s_\beta t)^{-1} sh))\right](x) &= \sum_{t\in H} f(s_\beta t)(x)  \alpha_{s_\beta}(g(t^{-1}s_\beta^{-1}sh))(x)\\
                &= \sum_{t\in H}f(s_\beta t)(x) g(s_\beta^{-1}st^{-1}h) (\sigma_{s_\beta}^{-1}(x))\\
                & = \sum_{t\in H} \phi_{x,s_\beta}(f)(t) \phi_{\sigma_{s_\beta}^{-1}(x), s_{\beta}^{-1}s}(g)(t^{-1}h)\\
                &= (\phi_{x,s_\beta}(f) * \phi_{\sigma_{s_\beta}^{-1}(x), s_\beta^{-1}s}(g))(h),
        \end{split}
    \end{align*}
    and we can express 
    $\phi_{x,s}(fg)(h) = \sum_{\beta\in I_H} (\phi_{x,s_\beta}(f) * \phi_{\sigma_\beta^{-1}(x), s_\beta^{-1}s}(g))(h)$
    for every $h\in H$. So we can write 
    $$\phi_{x,s}(fg) = \sum_{\beta\in I_H} \phi_{x,s_\beta}(f) * \phi_{\sigma_\beta^{-1}(x), s_\beta^{-1}s}(g),$$
    where the sum converges because
    \begin{align*}
        \begin{split}
            \lVert \sum_{\beta\in I_H} \phi_{x,s_\beta}(f) &* \phi_{\sigma_{s_\beta}^{-1}(x), s_\beta^{-1}s}(g)\rVert_1 \leq \sum_{\beta\in I_H} \lVert \phixs{x}{s_\beta}{f} \rVert_1 \lVert \phixs{\sigma_{s_\beta}^{-1}(x)}{s_\beta^{-1}s}{g} \rVert_1\\
            &\leq \sum_{\beta\in I_H} \left( \sum_{t\in H} |\phixs{x}{s_\beta}{f}(t)|\right) \lVert g\rVert = \lVert g\rVert \sum_{\beta\in I_H} \sum_{t\in H} |f(s_\beta t)(x)|\\
            &\leq \lVert g\rVert\sum_{\beta\in I_H} \sum_{t\in H} \lVert f(s_\beta t)\rVert_\infty =\lVert g\rVert \sum_{s\in G} \lVert f(s) \rVert_\infty = \lVert g\rVert \rVert f\rVert < \infty.
        \end{split}
    \end{align*}
    As $f\in\widetilde J$ we have that $\phixs{x}{s_\beta}{f} \in J$, so it follows that $ \phi_{x,s_\beta}(f) * \phi_{\sigma_\beta^{-1}(x), s_\beta^{-1}s}(g)$ also belongs to $J$ for every $\beta\in I_H$. Consequently, $\phixs{x}{s}{fg}$ is in $J$ as it is the limit of elements in $J$, which is a closed ideal. As this works for every $x\in X$ and every $s\in G$, we conclude that $fg \in \widetilde J$.

    Similarly, we prove that $gf$ belongs to $\widetilde J$. Let $x\in X$ and let $s\in G$, and take $h\in H$. We have that:
    \begin{align*}
        \begin{split}
            \phi_{x,s}(gf)(h)
            &= (gf)(sh)(x)=  \left[ \sum_{\beta \in I_H} \sum_{t\in H} g(s_\beta t) \alpha_{s_\beta t} (f((s_\beta t)^{-1}sh))\right](x)\\
            &=  \sum_{\beta \in I_H}\sum_{t\in H} g(s_\beta t) (x) \alpha_{s_\beta} (f(s_\beta^{-1}st^{-1}h)) (x)\\
            &= \sum_{\beta\in I_H} \sum_{t\in H} g(s_\beta t)(x) f(s_\beta^{-1}st^{-1}h) (\sigma_{s_\beta}^{-1}(x)) .
        \end{split}
    \end{align*}
    Fixing $\beta\in I_H$, we have that
    \begin{align*}
        \begin{split}
            \sum_{t\in H} g(s_\beta t)(x) f(s_\beta^{-1}st^{-1}h) (\sigma_{s_\beta}^{-1}(x))
            &= \sum_{t\in H} \phixs{x}{s_\beta}{g}(t) \phixs{\sigma_{s_\beta}^{-1}(x)}{s_\beta^{-1}s}{f}(t^{-1}h)\\
            &=(\phixs{x}{s_\beta}{g} * \phixs{\sigma_{s_\beta}^{-1}(x)}{s_\beta^{-1}s}{f})(h),
        \end{split}
    \end{align*}
    which belongs to $J$ because $g$ belongs to $\widetilde J$. So we can write
    $$\phi_{x,s}(gf) = \sum_{\beta \in I_H} \phixs{x}{s_\beta}{g} * \phixs{\sigma_{s_\beta}^{-1}(x)}{s_\beta^{-1}s}{f}, $$
    where the sum again converges by the same kind of computation as in the case above. So $\phixs{x}{s}{gf}\in J$ as it is the limit of elements in $J$. As this works for every $x\in X$ and $s\in G$, we have that $gf \in \widetilde J$. We can conclude then that $\widetilde J$ is a closed ideal of $\luno$.

    We only have left to check that $\widetilde J$ is not self-adjoint. As $J$ is not self-adjoint there is some $\xi \in J$ such that $\xi^* \not\in J$. We define $\widehat \xi : G \to C(X)$ the following way: for every $s\in G$, we set $\widehat\xi(s)$ as the constant function $\xi(s)$ if $s\in H$, and we set is as $0$ if not. As we have that $\xi\in\ell^1(H)$ it's direct to check that $\widehat\xi\in\luno$. Let's first check that $\widehat\xi$ is in $\widetilde J$. Let $x\in X$ and let $s\in G$, we have that
    \begin{equation*}
        \phixs{x}{s}{\widehat\xi}(h)
            = \widehat\xi(sh)(x)=
            \begin{cases}
                \xi(sh) &\text{if }sh\in H\\
                0 & \text{if } sh\not\in H
            \end{cases}
            =
            \begin{cases}
                (\mathbbm{1}_{s^{-1}} * \xi)(h) & \text{if } s\in H\\
                0 & \text{if } s\not\in H.\\
            \end{cases}
    \end{equation*}
    This means that if $s\not\in H$ then $\phixs{x}{s}{\widehat\xi} = 0$, so it is in $J$. If $s\in H$, we have then that $\phixs{x}{s}{\widehat\xi} = \mathbbm{1}_{s^{-1}}  *\xi$, which belongs to $J$ because $J$ is an ideal. We conclude then that $\widehat\xi$ is in $\widetilde J$. Let us check that $(\widehat{\xi})^*$ is not in $\widetilde J$. We take any $x\in X$ and consider $\phixs{x}{e}{(\widehat\xi)^*}$. For every $h\in H$ we have
    \begin{align*}
        \begin{split}
            \phixs{x}{e}{(\widehat\xi)^*}(h) 
            &= ((\widehat\xi)^*)(h)(x) = \alpha_{h}\left( \overline{\widehat\xi(h^{-1})} \right)(x) = \alpha_h\left( \overline{\xi(h^{-1})} \right) (x)\\
            &= \overline{\xi(h^{-1})}(\sigma_h^{-1}(x)) = \overline{\xi(h^{-1})} = \xi^*(h),
        \end{split}
    \end{align*}
    so $\phixs{x}{e}{(\widehat\xi)^*} = \xi^*$, which does not belong to $J$. So we obtain that $(\widehat\xi)^*$ is not in $\widetilde J$, and conclude that $\widetilde J$ is not self-adjoint.
\end{proof}

We need the following theorem, which is proved in \cite[Theorem 7.7.1]{Rudin}.
\begin{theorem}
    Let $H$ be a discrete and abelian group. If $H$ is infinite, there is a closed and non-self-adjoint ideal in $\ell^1(H)$.
\end{theorem}

\begin{lemma}\label{propAlmFree}
    Let $G$ be abelian and suppose that there is some $x\in X$ such that $G_x$ is infinite. Then $\luno$ has a closed and non-self-adjoint ideal.
\end{lemma}
\begin{proof}
    Let $x\in X$ such that $G_x$ is an infinite subgroup of $G$. By the theorem above, as $G_x$ is abelian, there is a closed and non-self-adjoint ideal in $\ell^1(G_x)$. Let $D$ be the closure of the orbit of $x$, which is a $G$-invariant and closed subset of $X$. Consider the restriction of the action to $D$ and consider the algebra $\ell^1(G\curvearrowright D)$. As $G$ is abelian, we have that every point in the orbit of $x$ has the same stabilizer subgroup, so $G_x$ is contained in the kernel of the action of $G$ on $D$. Using Lemma \refeq{propLiftIdeal} we have that there is a closed and non-self-adjoint ideal $\widetilde I$ in $\ell^1(G\curvearrowright D)$. We consider the contractive *-homomorphism $R_D: \luno \to \ell^1(G\curvearrowright D)$ given by $R_D(\sum_{s\in G}f_s\delta_s) = \sum_{s\in G}f_s|_D \delta_s$. As $R_D$ is surjective, we have that $I=R_D^{-1}(\widetilde I)$ is a closed and non-self-adjoint ideal of $\luno$.
\end{proof}

The above proposition basically tells us that, for abelian groups, an infinite point-stabilizer subgroup induces a closed but non-self-adjoint ideal in $\luno$. This allows us to obtain the following statement. We recall that, as always throughout the paper, $G$ is a countable infinite group acting on a compact Hausdorff space $X$.

\begin{theorem}
    Let $G$ be abelian and torsion-free. Then, the action $\action$ is free if and only if every closed ideal of $\luno$ is self-adjoint.
\end{theorem}
\begin{proof}
    If the action $\action$ is free then it is also r.t.f, so every closed ideal of $\luno$ is self-adjoint as a consequence of Corollary \refeq{libImpSelfAdj}. Suppose now that every closed ideal of $\luno$ is self-adjoint. If the action $\action$ is not free then there is some $x\in X$ such that $G_x \neq \{e\}$. As $G$ is torsion-free, $G_x$ must be infinite, so by Lemma \refeq{propAlmFree} there is a closed and non-self-adjoint ideal in $\luno$. This is a contradiction.
\end{proof}

\end{document}